\documentclass[12pt,reqno]{amsart}
\usepackage{amssymb,amsfonts,latexsym,amstext,amsmath,epsfig}
\usepackage[left=3cm,top=3cm,right=3cm,bottom=3cm]{geometry}

\newtheorem{theorem}{Theorem}[section]
\newtheorem{lemma}[theorem]{Lemma}
\numberwithin{equation}{section}

\newtheorem{corollary}[theorem]{Corollary}

\newcommand{\cw}{\textbf{cop-win}}
\newcommand{\kcw}{\textbf{$k$-cop-win}}
\newcommand{\universal}{\textbf{universal}}
\newcommand{\kd}{\textbf{$k$-dom}}
\newcommand{\N}{{\mathbb N}}
\newcommand{\Z}{{\mathbb Z}}
\newcommand{\R}{{\mathbb R}}
\newcommand{\eps}{\varepsilon}
\newcommand{\E}{\mathbb E}
\newcommand{\Bin}{\mathrm{Bin}}

\newcommand{\prob}[2][]{\mathbb{P}#1\left(#2#1\right)}

\begin{document}

\title{Almost all $k$-cop-win graphs contain a dominating set of cardinality $k$}

\author{Pawe{\l} Pra{\l}at}
\address{Department of Mathematics, Ryerson University, Toronto, ON, Canada, M5B 2K3}
\email{pralat@ryerson.ca}

\thanks{The author gratefully acknowledges support from NSERC and Ryerson}

\begin{abstract}
We consider $k$-cop-win graphs in the binomial random graph $G(n,1/2).$ It is known that almost all cop-win graphs contain a universal vertex. We generalize this result and prove that for every $k \in \N$, almost all $k$-cop-win graphs contain a dominating set of cardinality $k$. From this it follows that the asymptotic number of labelled $k$-cop-win graphs of order $n$ is equal to $(1+o(1)) (1-2^{-k})^{-k} {n \choose k} 2^{n^2/2 - (1/2-\log_2(1-2^{-k})) n}$.
\end{abstract}

\maketitle

\section{Introduction}

\emph{Cops and Robbers} is a vertex-pursuit game played on a reflexive graph. There are two players consisting of a set of \emph{cops} and a single \emph{robber}. The game is played over a sequence of discrete time-steps or \emph{rounds}, with the cops going first in each round. The cops and robber occupy vertices. When a player is ready to move in a round, they must move to a neighbouring vertex. Because of the loops, players can \emph{pass}, or remain on their own vertex. Observe that any subset of cops may move in a given round. The cops win if after some finite number of rounds, one of them can occupy the same vertex as the robber. This is called a \emph{capture}. The robber wins if he can evade capture indefinitely. A \emph{winning strategy for the cops} is a set of rules that, if followed, result in a win for the cops. A \emph{winning strategy for the robber} is defined analogously.

If we place a cop at each vertex, then the cops are guaranteed to win. Therefore, the minimum number of cops required to win in a graph $G$ is a well-defined positive integer, named the \emph{cop number} (or \emph{copnumber}) of the graph $G.$ We write $c(G)$ for the cop number of a graph $G$. If $c(G)\le k,$ then we say $G$ is $k$-\emph{cop-win}. In the special case $k=1,$ we say $G$ is \emph{cop-win} (or \emph{copwin}). Nowakowski and Winkler~\cite{nw}, and independently Quilliot~\cite{q}, considered the game with one cop only; the introduction of the cop number came in~\cite{af}. Many papers have now been written on cop number since these three early works; for more, see the monograph~\cite{bn}.

\bigskip

Since their introduction, the structure of cop-win graphs has been relatively well-understood. In~\cite{nw,q,q2} a kind of ordering of the vertex set---now called a cop-win or elimination ordering---was introduced which completely characterizes such graphs. If $u$ is a vertex, then the \emph{closed neighbour set of} $u,$ written $N[u],$ consists of $u$ along with the neighbours of $u.$ A vertex $u$ is a \emph{corner} if there is some vertex $v$, $v \neq u$, such that $N[u]\subseteq N[v].$ A graph is \emph{dismantlable} if some sequence of deleting corners results in the graph with a single vertex. For example, each tree is dismantlable, and more generally, so are chordal graphs (that is, graphs with no induced cycles of length more than three). To prove the latter fact, note that every chordal graph contains a vertex whose neighbour set is a clique; see, for example,~\cite{west}.

The following theorem gives the main results characterizing cop-win graphs.
\begin{theorem}\cite{nw,q,q2} \label{dis}
\begin{enumerate}
\item If $u$ is a corner of a graph $G,$ then $G$ is cop-win if and only if $G-u$ is cop-win.
\item A graph is cop-win if and only if it is dismantlable.
\end{enumerate}
\end{theorem}

\bigskip

We say that an event holds \emph{asymptotically almost surely} (\emph{a.a.s.}), if it holds with probability tending to one as $n$ tends to infinity. The probability of an event $A$ is denoted by $\prob{A}.$

Our goal is to investigate the structure of \emph{random cop-win }graphs. The random graph model we use is the familiar $G(n,1/2)$ probability space of all labelled graphs on $n$ vertices where each pair of vertices is joined with probability $1/2$, independent from the events for other pairs of vertices. Note that a given (labelled) graph $G$ on $n$ vertices occurs with probability
$$
\prob{G \in G(n,1/2)} = \left( \frac 12 \right)^{|E(G)|} \left(1- \frac 12 \right)^{{n \choose 2} - |E(G)|} = \left( \frac 12 \right)^{n \choose 2},
$$
which does not depend on $G$. Thus, $G(n,1/2)$ is in fact a uniform probability space over all labelled graphs on $n$ vertices. We heavily use this interpretation of $G(n,1/2)$ in the proof of our main result, Theorem~\ref{maincw}, stated below. We expect results analogous to Theorem~\ref{maincw} (that is, with 2 replaced by $1/p$) for $G(n,p)$, a natural generalization of $G(n,1/2)$, for other constants $p\in (0,1)$ and $p=p(n)$ tending to zero as $n \to \infty$. (The argument for $p=p(n)$ tending to one needs to be modified when the expected number of dominating sets of cardinality $k$ is $\Omega(1)$; see~\cite{Pawel}.) However, this seems not to be an interesting research direction in the theory of random graphs, where we usually focus on investigating typical properties that hold a.a.s.\ in $G(n,p)$. Therefore, studying bounds for the cop number that hold a.a.s.\ are of interest, and a number of papers have been published on this topic (see, for example~\cite{bkl, BPW, LP, pw, pw2}). We focus on $G(n,1/2)$ in this paper since it gives us a tool to investigate the typical structure of a $k$-cop-win graph. Therefore, from now on our probability space is always taken to be $G(n,1/2).$ 

\section{Main results}

A vertex set $S \subseteq V$ is a \emph{dominating set} for a graph $G = (V,E)$ if every vertex not in $S$ is adjacent to at least one member of $S$. A vertex $v \in V$ is \emph{universal} if it is joined to all others (that is, if $S=\{v\}$ is a dominating set). For a given $k \in \N$, let \kcw{} be the event that the graph is $k$-cop-win, let \kd{} be the event that there is a dominating set of cardinality $k$, and let \universal{} be the event that there is a universal vertex.  If a graph has a dominating set of cardinality $k$, then it is clearly $k$-cop-win. Hence, in a certain sense, graphs with dominating vertices of cardinality $k$ are the simplest $k$-cop-win graphs. The probability that a random graph is $k$-cop-win can be estimated as follows:
\begin{equation}
\prob{\mathrm{\kcw{} }} \ge
\prob{\kd{} } = (1+o(1)) \left( 1-2^{-k} \right)^{-k} {n \choose k} \left( 1-2^{-k} \right)^{n}.\label{lowerBound}
\end{equation}
Indeed, for any $S \subseteq V = \{1, 2, \ldots, n\}$ of cardinality $k$, let $A_S$ denote the event that $S$ is a dominating set. By the union bound,
$$
\prob{\kd{} } = \prob{\bigcup_{S} A_S} \le \sum_{S} \prob{A_S} = {n \choose k} \left( 1-2^{-k} \right)^{n-k}.  
$$ 
On the other hand, it follows from Bonferroni's inequality that
$$
\prob{\kd{} } = \prob{\bigcup_{S} A_S} \ge \sum_{S} \prob{A_S} - \sum_{S,T, S \neq T} \prob{A_S \cap A_T}.  
$$ 
Let $S,T \subseteq V$ be such that $|S|=|T|=k$ and $|S \cap T| = \ell$ for some $0 \le \ell < k$. The probability that a vertex $v \in V \setminus (S \cup T)$ is dominated by both $S$ and $T$ is equal to 
$$
\left( 1-2^{-\ell} \right) + \left(1-2^{-(k-\ell)}\right)^2 2^{-\ell} = 1-2^{-k} - 2^{-k} \left(1-2^{-(k-\ell)}\right) \le 1 - \frac 32 \cdot 2^{-k} < 1-2^{-k}.
$$
Hence, 
$$
\sum_{S,T, S \neq T} \prob{A_S \cap A_T} \le n^{2k} \left( 1 - \frac 32 \cdot 2^{-k} \right)^{n-2k} = o \left( \sum_{S} \prob{A_S} \right),
$$
and so the bound~(\ref{lowerBound}) holds.

\bigskip

Surprisingly, this lower bound is in fact the correct asymptotic value for $\prob{\kcw{} }.$ Our main result is the following theorem.
\begin{theorem}\label{maincw}
Let $k \in \N$. In $G(n,1/2),$ we have that
$$
\prob{\mathrm{\kcw{} } } = (1+o(1)) \left( 1-2^{-k} \right)^{-k} {n \choose k} \left( 1-2^{-k} \right)^{n} = \Theta \left(n^k \left( 1-2^{-k} \right)^{n} \right).
$$ 
\end{theorem}

We prove Theorem~\ref{maincw} in the next section. Using it, we derive the asymptotic number of labelled $k$-cop-win graphs.
\begin{corollary}
Let $k \in \N$. The number of $k$-cop-win graphs on $n$ labelled vertices is
$$
\prob{\mathrm{\kcw{} } } 2^{\binom{n}{2}}  = (1+o(1)) (1-2^{-k})^{-k} {n \choose k} 2^{n^2/2 - (1/2-\log_2(1-2^{-k})) n}.
$$
\end{corollary}
It also follows that almost all $k$-cop-win graphs contain a dominating set of cardinality $k$, a fact not obvious a priori.
\begin{corollary}
Let $k \in \N$. In $G(n,1/2),$ we have that
$$
\prob{\kd{} ~~|~~ \kcw{} } = 1-o(1).
$$
\end{corollary}

\bigskip

Let us mention that these observations generalize the result obtained in~\cite{bkp}.
\begin{theorem}[\cite{bkp}]\label{thm:old}
In $G(n,1/2),$ we have that
$$
\prob{\mathrm{\cw{} } } = (1+o(1))n2^{-n+1}.
$$
\end{theorem}
In~\cite{bkp}, it was conjectured that the theorem can be generalized, but this was left as an open problem. The proof of Theorem~\ref{thm:old} used the characterization of cop-win graphs (see Theorem~\ref{dis}). The limited current understanding of graphs with cop number two or higher was the main stumbling block in extending the result. For example, there are no elementary analogues of cop-win orderings for higher $k.$ An elimination ordering characterization of $k$-cop-win graphs for $k>1$ was given in~\cite{cm}, although it becomes  more complex as $k$ increases (in particular, a vertex ordering is provided but in the $(k+1)$th strong power of the graph). The proof of Theorem~\ref{maincw} avoids these complications and analyzes the winning strategy for the robber instead. Therefore, the proof techniques presented in this paper are quite different and might be of interest by themselves. 

\section{Proofs of main results}

We start this section with some notation that will be used in the proof. The \emph{neighbourhood} of a vertex $v$ is the set containing all neighbours of $v$ and is denoted by $N(v)$. We use $N^c(v) = V(G) \setminus (N(v) \cup \{v\})$ for the set of non-neighbours of $v$, and for every set $S \subseteq V(G)$, let $N^c(S) = \bigcap_{u \in S} N^c (u)$. For a given $k \in \N$, let 
$$
\delta_k = \min_{S \subseteq V, |S|=k} \left| N^c (S) \right|. 
$$
Note that $\delta_k = 0$ if and only if there exists a dominating set of cardinality $k$. 

\smallskip

To prove Theorem~\ref{maincw} we bound the probability of $\kcw{} $ for graphs with $\delta_k \ge 1$.  Since the proof for small values of $\delta_k$ has a different flavour than the one for large ones, we prove them independently.
\begin{theorem}\label{auxcw}
Let $k \in \N$. There exist $\xi > 0$ and $\eps > 0$ such that the following hold. In $G(n,1/2),$ we have that
\begin{enumerate}
\item[(a)]
$ \prob{\kcw{} \mbox{ and } 1 \le \delta_k \le \xi n } \le 2^{(\log_2(1-2^{-k})-\eps)n}$, and
\item[(b)]
$ \prob{\kcw{} \mbox{ and } \delta_k > \xi n } \le 2^{(\log_2(1-2^{-k})-\eps)n}$.
\end{enumerate}
\end{theorem}
\noindent Theorem~\ref{maincw} follows immediately from Theorem~\ref{auxcw} and~(\ref{lowerBound}).

\bigskip

Let $G$ be a random graph drawn from the $G(n,1/2)$ distribution. Our goal is to investigate the probability that $\delta_k \ge 1$ and that $c(G) \le k$. We show that this event holds \emph{with extremely small probability} (\emph{wesp}), which means that the probability it holds is at most $2^{(\log_2(1-2^{-k})-\eps)n}$ for some $\eps > 0$. Observe that if we can show that each of a polynomial number of events holds \emph{wesp}, then the same is true for the union of these events. 

\subsection{Proof of Theorem~\ref{auxcw}(a)} 

Let $S \subseteq V(G)$ be a set of vertices that dominates all but $\delta_k \ge 1$ vertices; let $v \in N^c (S)$ be some vertex \emph{not} dominated by $S$. To estimate the probability (from above) that $c(G)\le k$ we introduce a strategy for the robber and show that \emph{wesp} $k$ cops can win (against this given strategy). In order to motivate the desired structure of a graph $G$, we start with the outline of the strategy for the robber. However, at this point, it is not clear that this strategy can be successfully applied; this will be shown next. The outline is as follows:
\begin{enumerate}
\item [(a)] if cops move such that they are located in the whole set $S$, then the robber goes to $v$ ;
\item [(b)] if cops move to a proper subset $T$ of $S$ and no cop occupies $v$, then the robber goes to a vertex of $N^c(T) \cap N(v)$ that is adjacent to no cop;
\item [(b')] in particular, if no cop is located in $S \cup \{v\}$ (corresponding to $T = \emptyset$ above), then the robber goes to a vertex of $N(v)$ that is adjacent to no cop;
\item [(c)] if cops move to a proper subset $T$ of $S$ and some cop occupies $v$, then the robber goes to a vertex of $(N^c(T) \cap N^c(v)) \setminus S$ that is adjacent to no cop;
\item [(c')] in particular, if no cop is located in $S$ but some cop occupies $v$ (corresponding to $T = \emptyset$ above), then the robber goes to a vertex of $N^c(v) \setminus S$ that is adjacent to no cop.
\end{enumerate}
This outlines the strategy for the robber when the game is already on. The same applies to the beginning of the game; for example, if cops start the game such that they are located in the whole set $S$, then the robber starts by occupying $v$, etc. If at some point the game, the robber has more than one vertex to choose from, then she can make an arbitrary choice. This strategy is a greedy one that guarantees that the robber stays alive for at least one more round but does not ``think about the future''. If all the cops ``pause'', then we may assume that the robber ``pauses'' too. Hence, we may assume that in each round at least one cop moves.

\bigskip

The following technical lemma is the key observation. In order to simplify the notation, we consider the probability space $G(n+k+1,1/2)$ instead of $G(n,1/2)$.

\begin{lemma}\label{lem:key}
Let $S \subseteq V = \{ 1, 2, \ldots, n+k+1\}$ with $|S|=k$, and let $v \notin S$. Let $T \subset S$ with $0 \le |T| = \ell < k$. Let $x=x(n)$, $y=y(n)$, $z=z(n)$ be any deterministic functions such that $xn, yn, zn \in \Z$, $0 \le y, z \le 1$, and $0 \le x \le y$. Consider the following properties:
\begin{enumerate}
\item [(p1)] $v \in N^c(S)$,
\item [(p2)] $yn$ vertices of $W := V \setminus (S \cup \{v\})$ are adjacent to $v$,
\item [(p3)] $S$ dominates all but $zn$ vertices of $W$,
\item [(p4)] $xn$ vertices of $N(v)$ that are dominated by $S$ are not adjacent to any vertex of $T$,
\item [(p5)] the following is \emph{not} true: for every $U \subseteq W$ with $|U| = k-\ell$ and every $w \in W \setminus U$, there exists a vertex $w' \in N(v)$ adjacent to $w$ but not adjacent to any vertex of $T \cup U$.
\end{enumerate}

Let $G=(V,E) \in G(n+k+1,1/2)$. (Note that all the variables are chosen in advance before a random graph $G$ is generated.) Then, there exists $\xi \in (0,1/2]$ such that if $z \le \xi$, then properties (p1)-(p5) hold simultaneously \emph{wesp}. Moreover, the statement holds when in both (p4) and (p5), $N(v)$ is replaced by $N^c(v)$. 
\end{lemma}

Note the unusual statement of property (p5). This could be avoided but, since we are going to use the negation of (p5) later on, this is a convenient way of thinking about this property. Moreover, note that property (p1) holds with probability $(1/2)^k = \Theta(1)$ and is independent of the remaining properties. Hence, property (p1) could be clearly omitted, since it cannot help with showing that properties (p1)-(p5) hold simultaneously \emph{wesp}. Again, we keep it for convenience, as these are all the properties used later on.

\bigskip

Before we prove the lemma, let us show how the lemma implies Theorem~\ref{auxcw}(a). Let $G \in G(n+k+1,1/2)$. We want to estimate $\prob{\kcw{} \mbox{ and } \delta_k }$ for some $\delta_k \ge 1$. Let $S \subseteq V(G)$ be a set of vertices that dominates all but $\delta_k$ vertices, and let $v \in  N^c (S)$ be a vertex \emph{not} dominated by $S$. 

The robber will try to follow the strategy we outlined above. The goal is to show that if this strategy fails (that is, $k$ cops have a winning strategy), then $G$ satisfies properties (p1)-(p5) in Lemma~\ref{lem:key} for some specific choice of $S, v, T, x(n), y(n)$, and $z(n) = 1+\delta_k$. Since the number of possible choices for $S, v, T, x(n), y(n), z(n)$ to consider is at most $O(n^{k+1+3}) = n^{O(1)}$ (recall that $k=O(1)$) and for each fixed choice we get a statement that holds \emph{wesp}, we get the desired upper bound for $\prob{\kcw{} \mbox{ and } 1 \le \delta_k \le \xi n }$, where $\xi >0$ is a constant implied and guaranteed by Lemma~\ref{lem:key}.

\smallskip

Suppose that the cops start the game by going to the whole set $S$. The robber starts at $v$ and the desired position is achieved (see rule (a)). Suppose then that the cops start the game by going to set $T$, a proper subset of $S$ (possibly the empty set), and (perhaps) some vertices of $U \subseteq W = V(G) \setminus (S \cup \{v\})$. Let $w \in W \setminus U$ be any vertex. The robber can start at $w' \in N(v)$ that is not adjacent to any cop and the desired position is achieved (see rule (b)) unless property (p5) in Lemma~\ref{lem:key} holds and so (p1)-(p5) hold for some specific choice of parameters. (Note that $w'$ is adjacent to $w$ but this is not important at this point; $w$ was used to show the existence of $w'$ only.) Finally, suppose that the cops start the game by going to $T \subset U$ (perhaps $T = \emptyset$), to some $U \subseteq W$, and at least one cop starts at $v$. Again, we take an arbitrary vertex $w \in W \setminus U$ and use it to show that there exists a vertex $w' \in N^c(v)$ that is adjacent to no cop, the desired position in rule (c), unless (p1)-(p5) hold when in both (p4) and (p5), $N(v)$ is replaced by $N^c(v)$. 

\smallskip

Now, suppose that the robber is in her desired position and it is cops turn to move. We will show that regardless what they do, she will be able to move to another desired position. If cops move so that they occupy the whole $S$, the robber wants to move to $v$ (see rule (a) of the robber's strategy). Since there is no edge between $v$ and $S$, no cop occupied $v$ in the previous round. Hence, the robber must be currently in $N(v)$ (see rule (b)) and can easily move to $v$.

Suppose then that cops move to a proper subset $T$ of $S$ but not to $v$ (that is, some cops perhaps go to $U \subseteq W$). The robber is at $v$ (if cops were located in the whole $S$ in the previous round) or some vertex of $W$, and she wants to move to a vertex of $N^c(T) \cap N(v)$ that is adjacent to no cop (see rule (b)). If for every $T \subset S$, $U \subseteq W$ such that $|T \cup U|=k$, and $w \in W \setminus U$, there exists a vertex $w' \in N(v)$ adjacent to $w$ but not adjacent to any vertex of $T \cup U$, the robber can move from $w$ to $w'$ and survive for at least one more round, reaching another desired position. In other words, the robber's strategy fails for at this point if $G$ satisfies properties (p1)-(p5) in Lemma~\ref{lem:key} for some specific choice of $S, v, T, x(n), y(n), z(n)$. A symmetric argument can be used to analyze the case when cops move to $T \subset S$ and to $v$ (see rule (c)) to get a conclusion that if the robber's strategy fails because of that, then $G$ must satisfy properties (p1)-(p5), where $N(v)$ is replaced by $N^c(v)$. 

\bigskip

It remains to prove Lemma~\ref{lem:key}.

\begin{proof}[Proof of Lemma~\ref{lem:key}]
Since edges of a random graph are generated independently, the probability that properties (p2) and (p3) hold simultaneously is equal to
\begin{equation}\label{eq:bound_strong}
{n \choose yn} 2^{-n} {n \choose zn} \left(2^{-k}\right)^{zn} \left( 1 - 2^{-k} \right)^{(1-z)n} \le {n \choose yn} 2^{-n} {n \choose \xi n} \left( 1 - 2^{-k} \right)^{n},
\end{equation}
since $0 < \xi \le 1/2$. Using Stirling's formula ($n! = (1+o(1)) \sqrt{2 \pi n} (n/e)^n$) and taking the exponential part we obtain an upper bound of
$$
\exp \Big( (-y \ln y -(1-y) \ln (1-y) - \xi \ln \xi - (1-\xi) \ln (1-\xi) - \ln 2)n \Big) \left( 1 - 2^{-k} \right)^{n}.
$$
It is straightforward to see that $f(t) := - t \ln t - (1-t) \ln (1-t)$ tends to zero as $t \to 0$, and that $f(t)$ is maximized at $t=1/2$ giving $f(1/2) = \ln 2$. Moreover, if $y \le 1/2 - \eps_1$ or $y \ge 1/2+ \eps_1$ for some $\eps_1 > 0$, we get that $f(y) \le \ln 2 - \eps_2$ for some $\eps_2 = \eps_2(\eps_1) > 0$, and so properties (p2) and (p2) hold \emph{wesp} after taking $\xi$ sufficiently small so that, for example, $f(\xi) \le \eps_2 / 2$ (and so it is also the case that \emph{wesp} (p1)-(p5) hold). Indeed, for such choice of $y$ and $\xi$ we get a bound of
$$
\exp \Big( (\ln 2 - \eps_2 + \eps_2/2 - \ln 2)n \Big) \left( 1 - 2^{-k} \right)^{n} = 2^{(\log_2(1-2^{-k})-\eps_2/2)n}.
$$

Hence, we may assume that $1/2-\eps_1 \le y \le 1/2 + \eps_1$ for some $\eps_1 > 0$. (The constant $\eps_1$ can be made arbitrarily small by assuming that $\xi$ is small enough.) For this range of $y$, we lose negligible term by using slightly weaker bound (see~(\ref{eq:bound_strong})). The probability that properties (p2) and (p3) hold simultaneously is at most
$$
{n \choose yn} 2^{-n} {n \choose zn} \left(2^{-k}\right)^{zn} \left( 1 - 2^{-k} \right)^{(1-z)n} \le {n \choose \xi n} \left( 1 - 2^{-k} \right)^{n}.
$$
We need to consider (on top of that) property (p4) to obtain the desired bound for the event that (p1)-(p5) hold simultaneously. We first expose edges from $v$ to the vertices of $W$. For a vertex $u \in N(v)$, let $A(u)$ be the event that $u$ is dominated by $S$ and let $B(u)$ be the event that $u$ is nonadjacent to $T$. We can perform a ``double exposure'' for each vertex. First, we determine whether $u$ is dominated by $S$; $\prob{ A(u) } = (1-2^{-k})$. If this is the case, then we determine whether $u$ is also nonadjacent to $T$. This time,
$$
\eta := \prob{ B(u) | A(u) } = \frac {\prob{ B(u) \wedge A(u)}} {\prob{ A(u) }} = \frac {2^{-\ell}(1-2^{-(k-\ell)})}{1-2^{-k}}.
$$
(Note that, as expected, $\eta = 1$ if $T = \emptyset$; in any case, $\eta > 0$, since $T$ is a proper subgraph of $S$.) Since the number of vertices not dominated by $S$ is $zn \le \xi n$ and $|N(v)| = yn \ge (1/2-\eps_1)n$, the expected number of neighbours of $v$ that are dominated by $S$ but not adjacent to $T$ is at least
$$
\left( \frac 12 - \eps_1 - \xi \right) n \cdot \eta > \frac {\eta n}{3},
$$
provided $\xi$ (and so $\eps_1$ as well) are small enough. The events associated with two distinct vertices $u_1$ and $u_2$ are independent. We next use a consequence of Chernoff's bound (see e.g.~\cite[p.\ 27, Corollary~2.3]{JLR}), that for a random variable $X$ with the binomial distribution $\Bin(n,p)$ we have
\begin{equation}\label{chern}
\prob{ |X-\E [X]| \ge \eps \E [X]) } \le 2\exp \left( - \frac {\eps^2 \E [X]}{3} \right)  
\end{equation}
for  $0 < \eps < 3/2$. We get that the probability that (p4) holds for $x \le \eta / 6$ (conditioned on (p2) and (p3) holding) is at most
$$
\prob{ \Bin(n/3,\eta) \le \eta n/6 } \le \exp \left( - \frac {(1/2)^2 (\eta n / 3)}{3} \right) = \exp \left( - \frac {\eta}{36} n \right).
$$
Hence, if $x \le \eta/6$, then properties (p2)-(p4) hold simultaneously \emph{wesp}, after taking $\xi$ sufficiently small so that, for example, $f(\xi) \le \eta / 40$. (Again, this trivially implies that \emph{wesp} (p1)-(p5) hold.)

We now may assume, in addition to assuming that $1/2-\eps_1 \le y \le 1/2 + \eps_1$ for some $\eps_1 > 0$, that $x \ge \eta/6$. (As before, the constant $\eps_1$ can be made arbitrarily small by assuming that $\xi$ is small enough but recall that $\eta$ is not a function of $\xi$ and depends only on $k$ and $\ell$.) As before, we note that the probability that properties (p2) and (p3) hold is at most ${n \choose \xi n} \left( 1 - 2^{-k} \right)^{n}$, but this time we need to consider property (p5) to obtain the desired bound. In order to do it, we first expose edges from $S \cup \{v\}$ to $W$ (to estimate the probability that (p2) and (p3) hold) but do not yet expose any edge between vertices of $W$. Hence, we may estimate the probability that (p5) holds by exposing the edges of the subgraph induced by $W$, and all events are independent. The number of choices of $U$ and $w$ is $n^{O(1)}$. For a particular choice of $U$ and $w$, the probability that no suitable $v'$ can be found can be estimated by
$$
\left( 1 - 2^{-1} 2^{-(k-\ell)} \right)^{xn-(k-\ell)-1} \le \left( 1 - 2^{-(k-\ell)-1} \right)^{\eta n /7} = \exp \left( - \eps_3 n \right),
$$
where $\eps_3 = \eps_3(k, \ell) > 0$ and does not depend on $\xi$.
Indeed, with probability $2^{-1}$ a given candidate vertex $v'$ (neighbour of $v$, dominated by $S$ but not adjacent to $T$) is adjacent to $w$, and with probability $2^{-(k - \ell)}$ it is not adjacent to any vertex of $U$. The bound holds, since we have at least $xn-(k-\ell)-1 \ge \eta n /7$ candidates to test and the corresponding events are independent.  The properties (p2), (p3), and (p5) hold \emph{wesp} after taking $\xi$ sufficiently small so that, for example, $f(\xi) \le \eps_3 / 2$. Hence, \emph{wesp} (p1)-(p5) hold and the proof of the lemma is finished.
\end{proof}

\subsection{Proof of Theorem~\ref{auxcw}(b)}

Let $\xi >0$ be the constant yielded by Lemma~\ref{lem:key}. Let $G \in G(n,1/2)$. We want to estimate $\prob{\kcw{} \mbox{ and } \delta_k }$ for some $\delta_k \ge \xi n$. A vertex $b$ of $G$ is called \emph{dangerous} if there exists a \emph{threatening} set of $k$ vertices  $A = \{a_1, a_2, \ldots, a_k\}$, $b \notin A$, such that 
$$
\left| N^c(A) \cap N(b) \right| \le 2q,
$$
where $q := 3(k+1)/\xi = O(1)$. The robber might be afraid of going to dangerous vertices because if she does so, the cop can try to move to the threatening set and the robber has at most $2q$ vertices to escape to (possibly none!). However, if the total number of dangerous vertices is, say,  smaller than $q$, the robber can easily win. Indeed, regardless of the starting positions of the cops, at least $\xi n \ge q$ vertices are not dominated by the set of vertices occupied by the cops, and less than $q$ of them are dangerous. So the robber can start on a vertex that is not dangerous and not adjacent to any of the cops. We will call any such vertex \emph{safe} (with respect to the current position of cops).

Now, suppose that the robber occupies a safe vertex $v$. No matter where the cops move to, since $v$ is not dangerous, there are always more than $2q$ vertices not adjacent to any of the cops the robber can go to. Since the number of dangerous vertices is less than $q$, at least one of them is safe and she can go there. This shows that the robber will be able to stay away from dangerous vertices forever and visit only safe vertices. 

It remains to show that \emph{wesp} there are at least $q$ dangerous vertices. Suppose that $b_i$ is dangerous because of a threatening set $A_i = \{a_1^i, a_2^i, \ldots, a_k^i \}$ with
$$
\left| N^c(A_i) \right| = t_i \ge \delta_k \ge \xi n,
$$
but only at most $2q$ neighbours of $b_i$ are in $N^c(A_i)$, $i \in \{1, 2, \ldots, q\}$. Since there are only $O(n^{q(1 + k + 2q)}) = n^{O(1)}$ configurations to consider, we may focus on one. It may happen that $b_i = a_\ell^j$ for some $\ell$ and $j$. However, we can always select a subset of $b_i$'s of cardinality $q/(k+1) = 3/\xi$ without this property (that is, all vertices are distinct). This is a desired situation, since we want to be able to expose a lot of potential edges and insist on not generating any. To get an upper bound, we focus only on exposing edges between $b_i$'s and $N^c(A_i)$. Hence, we get that the probability that there are at least $q$ dangerous vertices is at most
$$
\prod_{i=1}^{q} \left( (1/2)^{t_i - O(1)} \right)^{\frac {q}{k+1}} \le \left( (1/2)^{\xi n - O(1)} \right)^{\frac {q}{k+1}} \le \left( (1/2)^{2 \xi n / 3} \right)^{\frac {3}{\xi}} = 2^{-2n},
$$
and so the property holds \emph{wesp} and the proof is finished.

\end{document}